\documentclass[twoside,12pt,a4paper]{amsart}
\usepackage{amscd,amsmath}
\usepackage{amssymb}
\usepackage{amsthm}
\usepackage{color}
\usepackage{hyperref}
\usepackage{enumerate}
\usepackage{geometry}
\theoremstyle{definition}
\newtheorem{theorem}{Theorem}[section]

\newtheorem{lemma}[theorem]{Lemma}
\newtheorem{corollary}[theorem]{Corollary}
\newtheorem{remark}[theorem]{Remark}

\numberwithin{equation}{section}
\newtheorem{problem}{Problem}

\newtheorem{acknowledgements}{Acknowledgement}

\def\<{\left < }
\def\>{\right >}
\def\({\left ( }
\def\){\right )}

\def\C2{${\bf C}^2$}

\begin{document}

\title[ First Chen Ineq. Gen.Warp.Prod.Subman. of Riemann Sp.Form - App.  ]{ First Chen Inequality for General Warped Product Submanifolds of a Riemannian Space Form and Applications} 

\author[A. Mustafa]{Abdulqader MUSTAFA}
\address{Department of Mathematics, Faculty of Arts and Science, Palestine Technical University, Kadoorei, Tulkarm, Palestine}
\email{abdulqader.mustafa@ptuk.edu.ps}
\author[C. \"Ozel]{Cenap \"OZEL}
\address{Department of Mathematics, Faculty of Science, King Abdulaziz University, 21589 Jeddah, Saudi Arabia}
\email{cozel@kau.edu.sa}
\author[A. Pigazzini]{Alexander PIGAZZINI}
\address{Mathematical and Physical Science Foundation, 4200 Slagelse, Denmark}
\email{pigazzini@topositus.com}
\author[R. Kaur]{Ramandeep KAUR}
\address{Department of Mathematics and Statistics, Central University of Punjab, Bathinda, Punjab-151 401, India}
\email{ramanaulakh1966@gmail.com}
\author[G. Shanker]{Gauree SHANKER}
\address{Department of Mathematics and Statistics, Central University of Punjab, Bathinda, Punjab-151 401, India}
\email{gauree.shanker@cup.edu.in}

\maketitle
\begin{abstract}
  In this paper, the first Chen inequality is proved for general warped product submanifolds in Riemannian space forms, this inequality involves  intrinsic invariants ($\delta$-invariant and sectional curvature) controlled by an extrinsic one (the mean curvature vector), which provides an answer for Problem \ref{prob3}. As a geometric application, this inequality is applied to derive a necessary condition for the immersed submanifold to be minimal in Riemannian space forms, which presents a partial answer for the well-known problem proposed by S.S. Chern, Problem \ref{prob6}. For further research directions, we address a couple of open problems; namely Problem \ref{ama1} and Problem \ref{pqm2}.\\



\end{abstract}


\sloppy
\section{Introduction}
Warped products have been playing some important roles in the theory of general relativity as they have been providing the best mathematical models of our universe for now; that is, the warped product scheme was successfully applied in general relativity and semi-Riemannian geometry in order to build basic cosmological models for the universe. For instance, the Robertson-Walker spacetime, the Friedmann cosmological models and the standard static spacetime are given as warped product manifolds. For more cosmological applications, warped product manifolds provide excellent setting to model spacetime near black holes or bodies with large gravitational force. For example, the relativistic model of the Schwarzschild spacetime that describes the outer space around a massive star or a black hole admits a warped product construction \cite{iijj77}. For this reason, in the years, many geometric aspects of the various types of warped product submanifods have been studied in various ambient spaces, (among others, see for example \cite{AKU17}, \cite{ffkk0}, \cite{2211gg}, \cite{CU}, \cite{aassll}, \cite{GenIneq}, \cite{abd}).

 Extrinsic and intrinsic Riemannian invariants have vast applications in many fields of science other than differential geometry. In particular, they are of considerable significance in general relativity \cite{99mmnn}. Classically, among extrinsic invariants, the shape operator and the squared mean curvature are the most important ones. Among the main intrinsic invariants, sectional, Ricci and scalar curvatures are the well-known ones, as well as $\delta$ invariant. So, based on Nash embedding theorem (see \cite{22vvuu}, \cite{88jj99}), our research programs is to search for control of extrinsic quantities in relation to intrinsic quantities of Riemannian manifolds via Nash's Theorem and to search for their applications \cite{2233ee}, \cite{55kk99}. Since it is an inevitable motivation, this was quite enough for Chen to address the following research problem:

\begin{problem}\label{prob3} \cite{aallr4}.
Establish simple relationships between the main extrinsic invariants and the main intrinsic invariants of a submanifold.
\end{problem}
 
Several famous results in differential geometry, such as isoperemetric inequality, Chern-Lashof's inequality and Gauss-Bonnet's theorem among others can be regarded as results in this respect. The current paper aims to continue this sequel of inequalities.   

It is well-known that \cite{66449d} the following two conditions are necessary for the immersion to be minimal in the Euclidean space $\mathbb{E}^m$:

\noindent{\bf Condition 1:} If $\varphi: M^n\rightarrow \mathbb{E}^m$ is a minimal immersion from a manifold of positive dimension into a Euclidean $m$-space, then $M^n$ is non-compact.

\noindent{\bf Condition 2:} If $\varphi: M^n\rightarrow \mathbb{E}^m$ is a minimal immersion from a manifold of positive dimension into a Euclidean $m$-space, then the Ricci tensor of $M^n$ is negative semi-definite.

S. S. Chern asked on page 13 of \cite{8822cc} to search for further necessary conditions on the Riemannian metric of a submanifold $M^n$ in order to admit an isometric minimal immersion into a Euclidean space $\mathbb{E}^m$. Later on, Chen materialized this goal for warped product submanifolds as the following:
\begin{problem}\label{prob6} \cite{2233ee}. Given a warped product $N_1\times _fN_2$, what are the necessary conditions for the warped product to admit a minimal isometric immersion in a Euclidean $m$-space $\mathbb{E}^m$ (or $\tilde M^m(c))?$
\end{problem}

In \cite{66449d}, Chen initiated a significant inequality in terms of the intrinsic invariant $\delta$-invariant. This celebrated inequality drew attention of several authors \cite{55kk99}. Motivated by the result of Chen, we construct a new general inequality in terms of $\delta$-invariants for general warped product submanifolds of Riemannian space forms.

We recall the following algebraic key lemma from \cite{66449d}
\begin{lemma}\label{48}
Let $\alpha_1, \alpha_2, \cdots, \alpha_n, \beta$ be $(n+1)~(n\ge 2)$ real numbers such that
$$ (\sum_{i=1}^n \alpha_i)^2= (n-1)  (\sum_{i=1}^n \alpha_i^2+\beta).$$
Then $2\alpha_1\alpha_2\ge \beta,$ with equality holds if and only if $\alpha_1+\alpha_2=\alpha_3=\cdots=\alpha_n$.
\end{lemma}

Also, we recall the following definition of the {\it Chen first invariant}, which is the main intrinsic invariant in our inequality
\begin{equation}\label{kk3}
\delta_{\tilde M^m}(x)=\tilde \tau (T_x\tilde M^m)-\inf \{\tilde K(\pi) : \pi \subset T_x\tilde M^m, x\in \tilde M^m, \dim \pi =2\}.
\end{equation}

The following  theorem was firstly proved for Riemannian submanifolds in real space forms by Chen in \cite{66449d}, it is known nowadays as the Chen first inequality. Ever since Chen published it, it has been extended for Riemannian submanifolds in various ambient manifolds \cite{ssee44}. By contrast, it was not proved for warped product submanifold in any ambient manifold. Indeed, for example in \cite{6677bb}, Chen establishes a general inequality for warped products in real space forms, but not addressing it in terms of his First Inequality.
\\
Thus, this work is devoted to prove it in the setting of warped product submanifolds.

The Chen first inequality was first proved in this form:
\begin{theorem}\label{pinc}
Let $M^n$ be an $n$-dimensional $(n\ge 2)$ submanifold of a Riemannian manifold $\tilde M^m(c)$ of constant sectional curvature $c$. Then 
\begin{equation}\label{100i3}
\inf K\ge \frac{1}{2}\biggl\{ \tau(T_xM^n) -\frac{n^2(n-2)}{n-1} ||\vec H||^2 -(n+1)(n-2)c\biggr\},
\end{equation}
where $K$ and $\tau(T_xM^n)$ are the sectional curvature and the scalar curvature of $M^n$, respectively, $x\in M^n$.

Equality holds if and only if, with respect to suitable orthonormal frame fields $e_1, \cdots, e_n, e_{n+1}, \cdots, e_m,$ the shape operators of $M^n$ in $\tilde M^m(c)$ take the following forms:
$$A_{e_{n+1}} =
 \begin{pmatrix}
 \mu_1 & 0&0 & \cdots &0 \\
0&\mu_2&0&\cdots &0\\
   0 &0 &\mu & \cdots &0\\
   \vdots & \vdots&\vdots & \ddots&\vdots\\

   0 &0& 0 &0 &\mu
  
 \end{pmatrix},~~~ \mu=\mu_1+\mu_2,
$$
$$A_{e_{r}} =
 \begin{pmatrix}
 h_{11}^r & h_{12}^r&0 & \cdots &0 \\
h_{12}^r&-h_{11}^r&0&\cdots &0\\
   0 &0 &0 & \cdots &0\\
   \vdots & \vdots&\vdots & \ddots&\vdots\\

   0 &0& 0 &0 &0
 \end{pmatrix},~~r=n+2,\cdots, m.
$$
\end{theorem}

The present paper is organized as follows: After the introduction, we present in Section 2, the preliminaries, basic definitions and formulas. In Section 3, we prove four preparatory basic lemmas, which are necessary and useful to the next section. In Section 4, we consider  general warped products in Riemannian space form to prove a general inequality involving $\delta$-invariant and the mean curvature vector, as an answer to Problem \ref{prob3}. In Section 5, we provide solutions to the Chern's Problem, 'whether we can find other necessary conditions for an isometric immersion to be minimal or not', Corollaries \ref{con11}, \ref{con22}. In the last section, we address two open problems related to the obtained results in this paper. 

\section{Preliminaries}

 Let $\tilde M^m$ be a smooth Riemannian manifold. If we choose two linearly independent tangent vectors $X,~Y\in T_x\tilde M^m$, then the {\it sectional curvature} of the $2$-plane $\pi$ spanned by $X$ and $Y$ is given in terms of the Riemannian curvature tensor $\tilde R$ by
\begin{equation}\label{E1}
\tilde K(X\wedge Y)=\frac{\tilde g(\tilde R(X,Y)Y, X)}{\tilde g(X,X) \tilde g(Y,Y)- (\tilde g(X,Y))^2},
\end{equation}
 where $\tilde g$ is the Riemannina metric tensor furnished on $\tilde M^m$.

In case that the $2$-plane $\pi$ is spanned by orthogonal unit vectors $X$ and $Y$ from the tangent space $T_x\tilde M^m,~x\in \tilde M^m$, the previous definition may be written as 
\begin{equation}\label{28}
\tilde K(\pi)=\tilde K_{\tilde M^m}(X\wedge Y)=\tilde g(\tilde R(X, ~Y)Y,~X).
\end{equation}

It is worth pointing out that, $\tilde K(\pi)$ is independent of the choice of the orthonormal basis $\{X,~Y\}$ of $\pi$, and it determines the Riemannian curvature tensor $\tilde R$ completely (O'Neill, 1983). In addition, if $\tilde K(\pi)$ is constant for all planes $\pi$ in $T_x\tilde M^m$ and for all points $x\in \tilde M^m$, say $\tilde K(\pi)=c$, then we call $\tilde M^m(c)$ a {\it real space form}. In fact, space forms are regarded as the simplest important class of Riemannian manifolds. Denoting by $\tilde M^m(c)$ a real space form of constant sectional curvature $c$, the curvature tensor of $\tilde M^m(c)$ is expressed as 
\begin{equation}\label{F1}
\tilde R(X,~Y)Z=c(\tilde g(Y,~Z)X-\tilde g(X,~Z)Y),
\end{equation}
for any $X,~Y,~Z\in \Gamma (T\tilde M^m(c))$.

In this context, we shall define another important Riemannian intrinsic invariant called the {\it scalar curvature} of $\tilde M^m$, and denoted by $\tilde\tau (T_x\tilde M^m)$, which, at some $x$ in $\tilde M^m$, is given by 
\begin{equation}\label{10}
\tilde\tau (T_x\tilde M^m)=\sum_{1\leq i\textless j\leq m}\tilde K_{ij} ,
\end{equation}
where $\tilde K_{ij}=\tilde K(e_i\wedge e_j)$. It is clear that, the first equality in \eqref{10} is congruent to the following equation which will be frequently used in subsequent sections
\begin{equation}\label{}
2\tilde\tau (T_x\tilde M^m)=\sum_{1\leq i\neq j\leq m}\tilde K_{ij}.
\end{equation}

For the subsequent sections we introduce another two concepts of Riemannian invariants. We first take an integer $k$ such that, $2\le k\le m$, then the {\it Riemannian invariant} \cite{55kk99}, denoted by $\Theta$, on a Riemannian $m$-manifold $\tilde M^m$ is defined by 
\begin{equation}\label{J1}
\Theta_k(x)=\frac{1}{(k-1)} \inf_{\Pi_k, e_u} \tilde Ric_{\Pi_k}(e_u), ~~~~~~x\in \tilde M^m,
\end{equation}
where $\Pi_k$ runs over all $k$-planes in $T_x\tilde M^m$ and $e_u$ runs over all unit vectors in $\Pi_k$. The second invariant is called the {\it Chen first invariant}, which is defined as
\begin{equation}\label{K1}
\delta_{\tilde M^m}(x)=\tilde\tau (T_x\tilde M^m)-\inf \{\tilde K(\pi) : \pi \subset T_x\tilde M^m, x\in \tilde M^m, \dim \pi =2\}.
\end{equation}

Next, we recall two important differential operators of a differentiable function $\psi$ on $\tilde M^m$; namely the {\it gradient} $\tilde\nabla \psi$ and the {\it Laplacian} $\Delta \psi$ of $\psi$, which are defined, respectively, as follows
\begin{equation}\label{L1}
\tilde g(\tilde\nabla \psi, X) = X(\psi),\;\;\;\;\;\; \Delta \psi =\sum_{i=1}^{m}((\tilde\nabla_{e_i}e_i)\psi- e_ie_i \psi)
\end{equation}
for any vector field $X$ tangent to $\tilde M^m$, where $\tilde\nabla$ denotes the Levi-Civita connection on $\tilde M^m$. 

In an attempt to construct manifolds of negative curvatures, R.L. Bishop and O'Neill \cite{ddyy7} introduced the notion of {\it warped product manifolds} as follows: Let $N_1$ and $N_2$ be two Riemannian manifolds with Riemannian metrics $g_{N_1}$ and $g_{N_2}$, respectively, and $f>0$ a $C^\infty$ function on $N_1$. Consider the product manifold $N_1\times N_2$ with its projections $\pi_1:N_1\times N_2\mapsto N_1$ and $\pi_2:N_1\times N_2\mapsto N_2$. Then, the {\it warped product} $\tilde M^m= N_1\times _fN_2$ is the Riemannian manifold $N_1\times N_2=(N_1\times N_2, \tilde g)$ equipped with a Riemannian structure such that 
 $\tilde g=g_{N_1} + f^2 g_{N_2}$.

A warped product manifold $\tilde M^m=N_1\times _fN_2$ is said to be {\it trivial} if the warping function $f$ is constant. For a nontrivial warped product $N_1\times _fN_2$, we denote by $\mathfrak{D}_1$ and $\mathfrak{D}_2$ the distributions given by the vectors tangent to leaves and fibers, respectively. Thus, $\mathfrak{D}_1$ is obtained from tangent vectors of $N_1$ via the horizontal lift and $\mathfrak{D}_2$ is obtained by tangent vectors of $N_2$ via the vertical lift. 

Now, let $\{e_1,\cdots,e_{n_1}, e_{n_1+1}, \cdots, e_m\}$ be local fields of orthonormal frame of $\Gamma (T\tilde M^m)$ such that $n_1$, $n_2$ and $m$ are the dimensions of $N_1$, $N_2$ and $\tilde M^m$, respectively. Then, for any Riemannian warped product $ \tilde M^m=N_1\times _fN_2$, it is a well-known fact that the sectional curvature and the warping function are related by (see, for example \cite{2233ee}, \cite{55kk99}, \cite{ssee44})
\begin{equation}\label{24}
\sum_{a=1}^{n_1}\sum_{A=n_1+1}^{m} \tilde K(e_a\wedge e_A)=\frac{n_2\Delta f}{f}.
\end{equation}

Here, it is well-known that the {\it second fundamental form} $h$ and the {\it shape operator} $A_\xi$ of $M^n$ are related by  
\begin{equation}\label{5}
g(A_\xi (X,Y)=g(h(X,Y),\xi)
\end{equation}
for all $X,Y\in \Gamma(TM^n)$ and $\xi\in \Gamma(T^\perp M^n)$ (for instance, see \cite{pom}, \cite{iijj77}).

Likewise, we consider a local field of orthonormal frames \footnotemark[\value{footnote}] \footnotetext{Throughout this work, $M^n=N_1\times _fN_2$ denotes for the isometrically immersed warped product submanifold in $\tilde M^m$. The numbers $m,~n,~n_1,$ and $n_2$ are the dimensions of $\tilde M^m$, $M^n$, $N_1$ and $N_2$, respectively.} 
$\{e_1, \cdots , e_n, e_{n+1}, \cdots, e_m\}$ on $\tilde M^m$, such that, restricted to $M^n$, $\{e_1, \cdots , e_n\}$ are tangent to $M^n$ and $\{e_{n+1}, \cdots, e_m\}$ are normal to $M^n$. Then, the {\it mean curvature vector} $\vec H(x)$ is introduced as \cite{pom}, \cite{iijj77}
\begin{equation}\label{8}
\vec H(x)=\frac{1}{n} \sum_{i=1}^{n} h(e_i, e_i),
\end{equation}

On one hand, we say that $M^n$ is a {\it minimal submanifold} of $\tilde M^m$ if $\vec H=0$. On the other hand, one may deduce that $M^n$ is totally umbilical in $\tilde M^m$ if and only if $h(X,Y)=g(X,Y) \vec H$, for any $X,~Y\in \Gamma (TM^n)$ \cite{yyhh88}. It is remarkable to note that the scalar curvature $\tau (x)$ of $M^n$ at $x$ is identical with the scalar curvature of the tangent space $T_xM^n$ of $M^n$ at $x$; that is, $\tau (x)= \tau (T_xM^n)$ \cite{2233ee}.

In this series, the well-known {\it equation} of {\it Gauss} is given by 
\begin{equation}\label{12}
R(X,Y,Z,W)= \tilde R(X,Y,Z,W)+ g(h(X, W), h(Y, Z)) - g(h(X, Z), h(Y, W)),
\end{equation}
for any vectors $X, ~Y,~ Z,~ W\in \Gamma (TM^n)$, where $\tilde R$ and $R$ are the curvature tensors of $\tilde M^m$ and $M^n$, respectively.

From now on, we refer to the coefficients of the second fundamental form $h$ of $M^n$ with respect to the above local frame by the following notation 
\begin{equation}\label{13}
h_{ij}^r=g(h(e_i,e_j),e_r),
\end{equation}
 where $i, j \in \{1, . . . , n\}$, and $r \in \{ n+1, . . . , m\}$. First, by making use of \eqref{13}, \eqref{12} and \eqref{28}, we get the following
\begin{equation}\label{}
 K (e_i \wedge e_j )= \tilde K (e_i \wedge e_j) + \sum_{r=n+1}^{m} (g(h_{ii}^{r}~e_r, h_{jj}^{r} ~e_r)- g(h_{ij}^{r}~e_r, h_{ij}^{r}~e_r)).
\end{equation}

Equivalently,
\begin{equation}\label{14}
 K (e_i \wedge e_j )= \tilde K (e_i \wedge e_j) + \sum_{r=n+1}^{m} (h_{ii}^{r} h_{jj}^{r}- (h_{ij}^{r})^2),
\end{equation}
where $\tilde K (e_i \wedge e_j)$ denotes the sectional curvature of the $2$-plane spanned by $e_i$ and $e_j$ at $x$ in the ambient manifold $\tilde M^m$. Secondly, by taking the summation in the above equation over the orthonormal frame of the tangent space of $M^n$, and due to \eqref{10}, we immediately obtain
\begin{equation}\label{15}
2\tau (T_xM^n)= 2\tilde\tau (T_xM^n) +n^2 ||\vec H||^2-||h||^2,
\end{equation}
where
\begin{equation}\label{16}
 \tilde\tau (T_xM^n)=\sum_{1\le i< j\le n}\tilde K (e_i \wedge e_j)
\end{equation}
denotes the scalar curvature of the n-plane $T_xM^n$ in the ambient manifold $\tilde M^m$.

\section{Basic Lemmas}

This section is devoted to prove key lemmas, these lemmas are essential to perform the proof of the main theorem in the next section.

Now, let us recall the following significant key result for warped product submanifolds $ M^n=N_1\times _fN_2$. It is well-known that the sectional curvature and the warping function are related by 
\begin{equation}\label{24}
\sum_{a=1}^{n_1}\sum_{A=n_1+1}^{n} K(e_a\wedge e_A)=\frac{n_2\Delta f}{f},
\end{equation}
where $\{e_1,\cdots,e_{n_1}, e_{n_1+1}, \cdots, e_n\}$ are local fields of orthonormal frame of $\Gamma (TM^n)$ such that $n_1$, $n_2$ and $n$ are the dimensions of $N_1$, $N_2$ and $M^n$, respectively. It is clear that $e_a\in \{e_1,\cdots, e_{n_1}\}$, and $e_A\in \{e_{n_1+1},\cdots, e_{n}\}.$

\begin{lemma}\label{92}
Let $\varphi$ be an isometric $\mathfrak{D}$-minimal immersion, from a warped product $M^n=N_1\times _fN_2$ into  a Riemannian manifold $\tilde M^m$. Then
$$||h(\mathfrak{D}_1, \mathfrak{D}_2)||^2= \tilde \tau (T_xM^n)-\tilde \tau (T_xN_1)-\tilde \tau (T_xN_2)-\frac {n_2 ~\Delta (f)}{f},$$
where $\mathfrak{D}_1$ and $\mathfrak{D}_2$ are the distributions of the first and the second factors of $N_1\times _fN_2$, respectively.
\end{lemma}
\begin{proof}
In virtue of the Gauss equation, we have
\begin{equation}\label{djnrg4654ygvf}
n^2||\vec H||^2=||h||^2+2\tau(T_xM^n) -2\tilde \tau(T_xM^n).
\end{equation}

Now, let $\{e_1, \cdots, e_{n_1}, e_{n_1+1}, \cdots, e_n=e_{n_1+n_2}\}$ and $\{e_{n+1}, \cdots, e_m\}$ be the local fields of orthonormal frames of $\Gamma(TM^n)$ and $\Gamma(T^\perp M^n)$, respectively, where $\{e_1, \cdots, e_{n_1}\}$ and $\{e_{n_1+1}, \cdots, e_n=e_{n_1+n_2}\}$ are the frames of $\Gamma(TN_1)$ and $\Gamma(TN_2)$, respectively.~Then, and without loss of generality, choose $e_{n+1}$ to be in the direction of the mean curvature vector $\vec H$.

Now, from $(\ref{16})$, we have 
\begin{equation}\label{fghjklwqer}
\tau \bigl(T_xM^n\bigr)=\sum_{1\le i<j\le n} K_{ij}=\sum_{a=1}^{n_1} \sum_{A=n_1+1}^{n}K_{aA}+\sum_{1\le a<b\le n_1} K_{ab}+\sum_{n_1+1\le A<B\le n} K_{AB}.
\end{equation}

Via (\ref{24}) and (\ref{16}), the above equation is congruent to  
\begin{equation}\label{frtrtyyywyywryewrw}
\tau \bigl(T_xM^n\bigr)=\frac{n_2 \Delta f}{f}
+\tau \bigl(T_xN_1\bigr)+ \tau \bigl(T_xN_2\bigr).
\end{equation}

In view of  (\ref{14}), it is common to have
\begin{equation}\label{vcvcnssnv577}
\tau \bigl(T_xN_1\bigr)=\sum_{r=n+1}^{m} \sum_{1\le a <b \le n_1} \biggl(h_{aa}^r h_{bb }^r-\bigl(h_{a b}^r\bigr)^2\biggr)+\tilde\tau \bigl(T_xN_2\bigr),
\end{equation}
and
\begin{equation}\label{3255555gsdfjhsdfjng}
\tau \bigl(T_xN_2\bigr)=\sum_{r=n+1}^{m} \sum_{n_1+1\le A <B\le n}\biggl(h_{AA}^rh_{BB}^r-\bigl(h_{AB}^r\bigr)^2\biggr)+\tilde \tau \bigl(T_xN_2\bigr).
\end{equation}

By $(\ref{djnrg4654ygvf})$-$(\ref{3255555gsdfjhsdfjng})$, one directly obtains

\begin{equation}
\biggl(\sum_{i=1}^{n}h_{ii}^{n+1}\biggr)^2=\sum_{r=n+1}^{m}\sum_{i=1}^{n}(h_{ii}^r)^2\newline +\sum_{r=n+1}^{m}\sum_{\substack {i,j=1 \\ i\neq j}}^{n}(h_{ij}^r)^2+\frac{2n_2\Delta f}{f}\notag $$$$+2\sum_{r=n+1}^{m}\sum_{1\le a<b\le n_1} \biggl(h_{aa}^r h_{bb}^r- (h_{ab}^r)^2\biggr) +2\sum_{r=n+1}^{m}\sum_{n_1+1\le A<B\le n} \biggl(h_{AA}^r h_{BB}^r- (h_{AB}^r)^2\biggr)\notag \\$$$$+2\biggl(\tilde\tau(T_xN_1)+\tilde\tau(T_xN_2)-\tilde\tau(T_xM^n)\biggr).~~~~~~~~~~~~~~~~~~~~~~~~~~~~~~~~~
\end{equation}

By rearranging the right hand side terms in an appropriate manner, we can obtain
\begin{equation}\label{1a}
\biggl(\sum_{i=1}^{n}h_{ii}^{n+1}\biggr)^2=\sum_{r=n+1}^{m}\sum_{a=1}^{n_1}(h_{aa}^r)^2+2\sum_{r=n+1}^{m}\sum_{1\le a<b\le n_1} h_{aa}^r h_{bb}^r$$$$+\sum_{r=n+1}^{m}\sum_{A=n_1+1}^{n}(h_{AA}^r)^2+2\sum_{r=n+1}^{m}\sum_{n_1+1\le A<B\le n} h_{AA}^r h_{BB}^r$$$$+\sum_{r=n+1}^{m}\sum_{\substack {i,j=1 \\ i\neq j}}^{n}(h_{ij}^r)^2-2\sum_{r=n+1}^{m}\sum_{1\le a<b\le n_1}  (h_{ab}^r)^2 -2\sum_{r=n+1}^{m}\sum_{n_1+1\le A<B\le n}  (h_{AB}^r)^2 \\$$$$+\frac{2n_2\Delta f}{f} \\+2\biggl(\tilde\tau(T_xN_1)+\tilde\tau(T_xN_2)-\tilde\tau(T_xM^n)\biggr).~~~~~~~~~~~~~~~~~~~~~~~~~~~~~~~~~
\end{equation}

Since the the immersion is a $\mathfrak{D}$-minimal warped product Riemannian manifold, we can write
$$\sum_{r=n+1}^{m}\sum_{a=1}^{n_1}(h_{aa}^r)^2+2\sum_{r=n+1}^{m}\sum_{1\le a<b\le n_1} h_{aa}^r h_{bb}^r=0,$$
$$\sum_{r=n+1}^{m}\sum_{A=n_1+1}^{n}(h_{AA}^r)^2+2\sum_{r=n+1}^{m}\sum_{n_1+1\le A<B\le n} h_{AA}^r h_{BB}^r=\biggl(\displaystyle\sum_{i=1}^{n}h_{ii}^{n+1}\biggr)^2$$
and
$$\sum_{r=n+1}^{m}\sum_{\substack {i,j=1 \\ i\neq j}}^{n}(h_{ij}^r)^2-2\sum_{r=n+1}^{m}\sum_{1\le a<b\le n_1}  (h_{ab}^r)^2 -2\sum_{r=n+1}^{m}\sum_{n_1+1\le A<B\le n}  (h_{AB}^r)^2$$$$=2\displaystyle\sum_{r=n+1}^{m}\sum_{a=1}^{n_1}\sum_{A=n_1+1}^{n}(h_{aA}^r)^2.$$

By substituting the above three equations in (\ref{1a}), we immediately reach
$$2\sum_{r=n+1}^{m}\sum_{a=1}^{n_1}\sum_{A=n_1+1}^{n}(h_{aA}^r)^2=-\frac{2n_2\Delta f}{f}-2\biggl(\tilde\tau(T_xN_1)+\tilde\tau(T_xN_2)-\tilde\tau(T_xM^n)\biggr).$$

 This gives the assertion.
\end{proof}

For simplicity's sake, we present three computational lemmas in this section. Once they are verified, the proofs of inequalities in the next section become straightforward.

The following result is useful in proofs of the next inequality.
\begin{lemma}\label{djei93}
Let $\varphi$ be an isometric $\mathfrak{D}$-minimal immersion, from a warped product $M^n=N_1\times _fN_2$ into  a Riemannian manifold $\tilde M^m$. Then, we have
\begin{equation}\label{93928vkg0}
\frac{1}{2}\sum_{\substack{i,j=1\\i\neq j}}^n(h_{ij}^{n+1})^2+\frac{1}{2}\sum_{r=n+2}^{m}~\sum_{i,j=1}^n(h_{ij}^r)^2+\sum_{r=n+2}^m h_{11}^rh_{22}^r-\sum_{r=n+1}^m(h_{12}^r)^2=$$$$
\frac{1}{2}\sum_{\substack {i,j=3\\i\neq j}}^n(h_{ij}^{n+1})^2+\frac{1}{2}\sum_{r=n+2}^{m}~\sum_{i,j=3}^n(h_{ij}^r)^2+
\frac{1}{2}\sum_{r=n+2}^m(h_{11}^r+h_{22}^r)^2+\sum_{r=n+1}^m
~\sum_{j=3}^n\biggl((h_{1j}^r)^2+ (h_{2j}^r)^2\biggr),
\end{equation}
where $n$ and $m$ are the dimensions of $M^n$ and $\tilde M^m$, respectively.
\end{lemma}
\begin{proof}
If we start from the left hand side of $(\ref{djei93})$, then the first two terms can be respectively expanded in this way
\begin{equation}\label{sivng}
\frac{1}{2}\sum_{\substack{i,j=1\\i\neq j}}^n(h_{ij}^{n+1})^2=\frac{1}{2}\sum_{\substack{i,j=3\\i\neq j}}^n(h_{ij}^{n+1})^2+\sum_{j=3}^n(h_{1j}^{n+1})^2+(h_{12}^{n+1})^2+\sum_{j=3}^n(h_{2j}^{n+1})^2,
\end{equation}
and
\begin{equation}\label{acbuer}
\frac{1}{2}\sum_{r=n+2}^{m}~\sum_{i,j=1}^n(h_{ij}^r)^2=\frac{1}{2}\sum_{r=n+2}^{m}~\sum_{i,j=3}^n(h_{ij}^r)^2+\sum_{r=n+2}^{m}
~\sum_{j=3}^n(h_{1j}^r)^2+\sum_{r=n+2}^{m}~\sum_{j=3}^n(h_{2j}^r)^2
$$$$+\sum_{r=n+2}^m(h_{12}^r)^2+\frac{1}{2}\sum_{r=n+2}^{m}\bigg((h_{11}^r)^2+(h_{22}^r)^2\bigg).
\end{equation}

Using equations $(\ref{sivng})$ and $(\ref{acbuer})$ to substitute the first two terms on the left hand side of $(\ref{djei93})$, taking into consideration the following relations
$$\sum_{r=n+2}^m h_{11}^rh_{22}^r+\frac{1}{2}\sum_{r=n+2}^{m}\bigg((h_{11}^r)^2+(h_{22}^r)^2\bigg)=\frac{1}{2}\sum_{r=n+2}^m(h_{11}^r+h_{22}^r)^2,$$

$$\sum_{j=3}^n(h_{1j}^{n+1})^2+\sum_{r=n+2}^{m}~\sum_{j=3}^n(h_{1j}^r)^2
+\sum_{j=3}^n(h_{2j}^{n+1})^2+\sum_{r=n+2}^{m}~\sum_{j=3}^n(h_{2j}^r)^2
$$$$=\sum_{r=n+1}^m~\sum_{j=3}^n\biggl((h_{1j}^r)^2+ (h_{2j}^r)^2\biggr)$$
and
$$(h_{12}^{n+1})^2+\sum_{r=n+2}^m(h_{12}^r)^2=\sum_{r=n+1}^m(h_{12}^r)^2,$$
the right hand side of $(\ref{djei93})$ follows immediately, and completes the proof.
\end{proof}

Finally, for the  inequality in the next section, we prove the following computational lemma.
\begin{lemma}\label{kdiwi84cb}
Let $\varphi$ be an isometric $\mathfrak{D}$-minimal immersion, from a warped product $M^n=N_1\times _fN_2$ into  a Riemannian manifold $\tilde M^m$. Then, the following holds
\begin{equation}\label{fwgb23463}
\frac{1}{2}\sum_{\substack {i,j=3\\i\neq j}}^n(h_{ij}^{n+1})^2+\frac{1}{2}\sum_{r=n+2}^{m}~\sum_{i,j=3}^n(h_{ij}^r)^2+\sum_{r=n+1}^m
~\sum_{j=3}^n\biggl((h_{1j}^r)^2+ (h_{2j}^r)^2\biggr)=$$$$
\frac{1}{2}\sum_{\substack{a,b=3\\a\neq b}}^{n_1}(h_{ab}^{n+1})^2+\frac{1}{2}\sum_{\substack{A,B=n_1+1\\A\neq B}}^{n}(h_{AB}^{n+1})^2+\frac{1}{2}\sum_{r=n+2}^{m}~\sum_{a,b=3}^{n_1}(h_{ab}^r)^2+\frac{1}{2}\sum_{r=n+2}^{m}~\sum_{A,B=n_1+1}^{n}(h_{AB}^r)^2$$$$+\sum_{r=n+1}^m~\sum_{a=3}^{n_1}\biggl((h_{1a}^r)^2+ (h_{2a}^r)^2\biggr)+\sum_{r=n+1}^{m}~\sum_{a=1}^{n_1}
~\sum_{A=n_1+1}^n(h_{aA}^r)^2,
\end{equation}
where $n_1,~n$ and $m$ are the dimensions of $N_1,~M^n$ and $\tilde M^m$, respectively.
\end{lemma}
\begin{proof}
It is not difficult to expand the following expressions as
\begin{equation}\label{5ujcm}
\frac{1}{2}\sum_{\substack{i,j=3\\i\neq j}}^n(h_{ij}^{n+1})^2=\frac{1}{2}\sum_{\substack{a,b=3\\a\neq b}}^{n_1}(h_{ab}^{n+1})^2+\frac{1}{2}\sum_{\substack{A,B=n_1+1\\A\neq B}}^n(h_{AB}^{n+1})^2+\sum_{a=3}^{n_1}~\sum_{A=n_1+1}^n(h_{aA}^{n+1})^2,
\end{equation}
and 
\begin{equation}\label{xogis}
\frac{1}{2}\sum_{r=n+2}^m~\sum_{i,j=3}^n(h_{ij}^{r})^2=\frac{1}{2}\sum_{r=n+2}^m~\sum_{a,b=3}^{n_1}(h_{ab}^{r})^2+\frac{1}{2}\sum_{r=n+2}^m~\sum_{A,B=n_1+1}^n(h_{AB}^{r})^2
$$$$+\sum_{r=n+2}^m~\sum_{a=3}^{n_1}~\sum_{A=n_1+1}^n(h_{aA}^{r})^2.
\end{equation}

If we use the above two equations to substitute the first two terms in the left hand side of $(\ref{fwgb23463})$, taking into account the following equation
$$\sum_{r=n+1}^m
~\sum_{j=3}^n\biggl((h_{1j}^r)^2+ (h_{2j}^r)^2\biggr)+\sum_{a=3}^{n_1}~\sum_{A=n_1+1}^n(h_{aA}^{n+1})^2
+\sum_{r=n+2}^m~\sum_{a=3}^{n_1}~\sum_{A=n_1+1}^n(h_{aA}^{r})^2
=$$$$\sum_{r=n+1}^m~\sum_{a=3}^{n_1}\biggl((h_{1a}^r)^2+ (h_{2a}^r)^2\biggr)+\sum_{r=n+1}^{m}~\sum_{a=1}^{n_1}
~\sum_{A=n_1+1}^n(h_{aA}^r)^2,$$
then the right hand side of $(\ref{fwgb23463})$ automatically follows. 
\end{proof}

\section{ First Chen  Inequality of General Warped  Product Submanifolds of a Riemannian Space Form}

\begin{theorem}\label{315}
Let $\varphi :M^n=N_1\times _fN_2 \longrightarrow \tilde M^m(c)$ be an isometric immersion of a warped product submanifold $M^n$ into a Riemannian space form $\tilde M^m(c)$. Then, for each point $x\in M^n$ and each plane section $\pi_i \subset T_xN_i^{n_i}$, $n_i\ge 2$, for $i=1,~2$, we have: 
 \begin{itemize}
\item[(i)] if $\pi_1 \subset T_xN_1$, then 
\begin{equation}\label{10073}
\delta_{N_1^{n_1}}(x)\le \frac{n^2}{2} ||\vec H||^2-\frac{n_2 \Delta f}{f}+\frac{1}{2}n_1(n_1+2n_2-1)c-c;
\end{equation}
\item[(ii)] if $\pi_2 \subset T_xN_2$, then 
\begin{equation}\label{}
\delta_{N_2^{n_2}}(x)\le \frac{n^2}{2} ||\vec H||^2-\frac{n_2 \Delta f}{f}+\frac{1}{2}n_2(n_2+2n_1-1)c-c.
\end{equation}
\end{itemize}

Equalities of the above two inequalities hold at $x\in M^n$ if and only if there exists an orthonormal basis $\{e_1, \cdots, e_n\}$ of $T_xM^n$ and an orthonormal basis $\{e_{n+1}, \cdots, e_m\}$ of $T_x^\perp M^n$ such that (a) $\pi = Span \{e_1, e_2\}$ and (b) the shape operators take the following forms:
\begin{itemize}
\item[(~$\grave{i}$)] If $\pi_1 \subset T_xN_1$, then for $r=n+1$, we have
$$A_{e_{n+1}} =
 \begin{pmatrix}
  \mu_1 & h_{12}^{n+1} & 0 & \cdots & 0_{1n_1} & \vline & 0_{1n_1+1} & \cdots & 0_{1n} \\
  h_{21}^{n+1} & \mu_2 & 0 & \cdots & \vdots & \vline & \vdots & \cdots & \vdots  \\
  0 & 0 & \mu & \cdots & \vdots & \vline& \vdots &  \cdots& \vdots \\
  \vdots & \vdots & \vdots & \ddots & \vdots & \vline & \vdots & \cdots & \vdots \\
  0_{n_11} & 0 & 0 & \cdots & \mu & \vline & 0_{n_1n_1+1} & \cdots & 0_{n_1n} \\
\hline
  0_{n_1+11} & \cdots & \cdots & \cdots & 0_{n_1+1n_1} & \vline & h^{n+1}_{n_1+1n_1+1} & \cdots & h^{n+1}_{n_1+1n} \\
  \vdots & \ddots & \ddots & \ddots & \vdots & \vline & \vdots & \ddots & \vdots \\
  0_{n1} & \cdots & \cdots & \cdots & 0_{nn_1} & \vline & h^{n+1}_{nn_1+1} & \cdots & h^{n+1}_{nn} 
 \end{pmatrix}
,$$
$\mu=\mu_1+\mu_2.$

 If $r\in \{n+2, \cdots, m\}$, then we have 
$$A_{e_r} = 
\begin{pmatrix}
  h_{11}^{r} & h_{12}^{r} & 0 & \cdots & 0_{1n_1} & \vline & 0_{1n_1+1} & \cdots & 0_{1n} \\
  h_{21}^{r} & - h_{11}^{r} & 0 & \cdots & \vdots & \vline & \vdots & \cdots & \vdots  \\
  0 & 0 & 0_{33} & \cdots & \vdots & \vline& \vdots & \cdots & \vdots \\
  \vdots & \vdots & \vdots & \cdots & \vdots & \vline & \vdots &\cdots  & \vdots \\
  0_{n_11} & 0 & 0 & \cdots & 0_{n_1n_1} & \vline & 0_{n_1n_1+1} & \cdots & 0_{n_1n} \\
\hline
  0_{n_1+11} & \cdots & \cdots & \cdots & 0_{n_1+1n_1} & \vline & h^r_{n_1+1n_1+1} & \cdots & h^r_{n_1+1n} \\
  \vdots & \ddots & \ddots &\ddots  & \vdots & \vline & \vdots & \ddots & \vdots \\
  0_{n1} & \cdots & \cdots & \cdots &0_{nn_1} & \vline & h^r_{nn_1+1} & \cdots & h^r_{nn} 
 \end{pmatrix}.
$$
\item[$(\grave{ii})$] If $\pi_2 \subset T_xN_2$, then for $r=n+1$, we have
$$A_{e_{n+1}} =~~~~~~~~~~~~~~~~~~~~~
~~~~~~~~~~~~~~~~~~~~~~
~~~~~~~~~~~~~~~~~~~~~~~
~~~~~~~~~~~~~~~~~~~~~~~
~~~~~~~~~~~~~~~~~~~~~~~~
~~~~~~~~~~~~~~~~~~~~~~~~~
~~~~~~~~~~~~~~~~~~~~~~~~~~
~~~~~~~~~~~~~~~~~~~~~~~~~~~
~~~~~~~~~~~~~~~~~~~~~~~~~~~~~~~~
~~~~~~~~~~~~~~~~~~~~~~~~~~~~~~~~~~$$$$ \begin{pmatrix}
 h^{n+1}_{11} & \cdots& \cdots & h^{n+1}_{1n_1} &\vline& 0_{1n_1+1} & \cdots & \cdots & \cdots &0_{1n}\\
\vdots&\ddots&&\vdots &\vline&\vdots&\ddots&\ddots&\ddots&\vdots\\
   \vdots & &\ddots & \vdots &\vline &\vdots & \ddots &\ddots  &\ddots  &\vdots\\
   h^{n+1}_{n_11} & \cdots&\cdots & h^{n+1}_{n_1n_1} &\vline& 0_{n_1n_1+1} & \cdots & \cdots & \cdots & 0_{n_1n}\\
\hline
   0_{n_1+11} & \cdots& \cdots & 0_{n_1+1n_1} &\vline &\mu_1 & h^{n+1}_{n_1+1n_1+2} & 0 & \cdots & 0_{n_1+1n}\\
    \vdots &\ddots & \ddots & \vdots  &\vline &h^{n+1}_{n_1+2n_1+1} & \mu_2 & 0 & \cdots & \vdots \\
  \vdots &\ddots  &\ddots & \vdots &\vline& 0 & 0 & \mu & \cdots & \vdots \\
  \vdots &\ddots  &\ddots & \vdots & \vline & \vdots & \vdots & 0 &\ddots &0\\
    0_{n1} &\cdots &\cdots & 0_{nn_1} &\vline &0_{nn_1+1} & 0 & \cdots&0 & \mu 
 \end{pmatrix}
,$$
where $\mu=\mu_1+\mu_2.$ 

If $r\in \{n+2, \cdots, m\}$, then we have 
$$A_{e_{r}}=~~~~~~~~~~~~~~~~~~~~~~~~~~~~~~~~~~~~~~
~~~~~~~~~~~~~~~~~~~~~~~~~~~~
~~~~~~~~~~~~~~~~~~~~~~~~~~~~~~~~~~~~~~~~~~~~~~~
~~~~~~~~~~~~~~~~~~~~~~~~~~~~
~~~~~~~~~~~~~~~~~~~~~~~~~$$$$
 \begin{pmatrix}
 h^r_{11} & \cdots& \cdots & h^r_{1n_1} &\vline& 0_{1n_1+1} & \cdots & \cdots & \cdots &0_{1n}\\
\vdots&\ddots&&\vdots &\vline&\vdots&\ddots&\ddots&\ddots&\vdots\\
   \vdots & &\ddots & \vdots &\vline &\vdots & \ddots &\ddots  &\ddots  &\vdots\\
   h^r_{n_11} & \cdots&\cdots & h^r_{n_1n_1} &\vline& 0_{n_1n_1+1} & \cdots & \cdots & \cdots & 0_{n_1n}\\
\hline
   0_{n_1+11} & \cdots& \cdots & 0_{n_1+1n_1} &\vline &h_{n_1+1n_1+1}^{r} &  h_{n_1+1n_1+2}^{r}& 0 & \cdots & 0_{n_1+1n}\\
    \vdots &\ddots & \ddots & \vdots  &\vline &h_{n_1+2n_1+1}^{r} & -h_{n_1+1n_1+1}^{r} & 0 & \cdots & \vdots \\
  \vdots &\ddots  &\ddots & \vdots &\vline& 0 & 0 & 0 & \cdots & \vdots \\
  \vdots &\ddots  &\ddots & \vdots & \vline & \vdots & \vdots & 0 &\ddots &0\\
    0_{n1} &\cdots &\cdots & 0_{nn_1} &\vline &0_{nn_1+1} & 0 & \cdots&0 & 0 
 \end{pmatrix}
.$$
\item[$(\grave{iii})$] If the equality of $(i)$ or $(ii)$ holds, then $N_1\times_fN_2$ is mixed totally geodesic in $\tilde M^m(c)$. Moreover, $N_1\times_fN_2$ is both $\mathfrak{D}_1$-minimal and $\mathfrak{D}_2$-minimal. Thus, $N_1\times_fN_2$ is a minimal warped product submanifold in $\tilde M^m(c)$. 
\end{itemize}
\end{theorem}

\begin{proof}
For $x\in M^n$, let $\pi_1\subset T_xN_1$ be a $2$-plane. We choose an orthonormal basis $\{e_1, \cdots, e_{n_1}, e_{n_1+1}, \cdots, e_{n}\}$ of $T_xM^n$, where $\{e_1, \cdots, e_{n_1}\}$ is an orthonormal basis for $T_xN_1$ and $\{e_{n_1}, e_{n_1+1}, \cdots, e_{n}\}$ is for $T_xN_2$. Hence, $\{e_{n+1}, \cdots, e_m\}$ is an orthonormal basis of $T_x^\perp M^n$. First, put $\pi_1 = Span \{e_1, e_2\}$ such that the normal vector $e_{n+1}$ is in the direction of the mean curvature vector $\vec H$. By $(\ref{12})$ and $(\ref{F1})$ we have
\begin{equation}
n^2 ||\vec H||^2= 2 \tau (T_xM^n)+ ||h||^2-n(n-1)c.\label{}
\end{equation}
Equivalently, 
$$\biggl(\sum_{a=1}^{n_1}h_{aa}^{n+1}\biggr)^2=2\tau (T_xM^n)+||h||^2-n(n-1)c-\biggl(\sum_{A=n_1+1}^nh_{AA}^{n+1}\biggr)^2-2\sum_{a=1}^{n_1}~\sum_{A=n_1+1}^nh_{aa}^{n+1}h_{AA}^{n+1}.$$

Putting 
\begin{equation}\label{dme}
\Upsilon_1 = 2\tau (T_xM^n)- \frac{n_1-2}{n_1-1}  \biggl(\sum_{a=1}^{n_1}h_{aa}^{n+1}\biggr)^2~~~~~~~~~~~~~~~~~~~~~~~~~~~~~~~~$$$$-\biggl(\sum_{A=n_1+1}^nh_{AA}^{n+1}\biggr)^2-2\sum_{a=1}^{n_1}\sum_{A=n_1+1}^nh_{aa}^{n+1}h_{AA}^{n+1}-n(n-1)c.\label{}
\end{equation}

Thus, from the above two equations we may write
\numberwithin{equation}{section}
\begin{equation}
\biggl(\sum_{a=1}^{n_1}h_{aa}^{n+1}\biggr)^2= (n_1-1) \biggl(\Upsilon_1 + ||h||^2\biggr),\label{}
\end{equation}
i.e.,
\numberwithin{equation}{section}
\begin{equation}
\biggl(\sum_{a=1}^{n_1} h_{aa}^{n+1}\biggr)^2= (n_1-1)\biggl (\Upsilon_1 + \sum_{a=1}^{n_1} (h_{aa}^{n+1})^2+\sum_{A=n_1+1}^{n} (h_{AA}^{n+1})^2$$$$~~~~~~~~~~+\sum_{\substack{i,j=1\\i\neq j}}^n(h_{ij}^{n+1})^2+\sum_{r=n+2}^{m}~\sum_{i,j=1}^n(h_{ij}^r)^2\biggr).\label{e}
\end{equation}

Applying Lemma \ref{48} on the above equation for
$$\alpha_a=h_{aa}^{n+1},~~~ \forall~ a\in \{1, \cdots, n_1\}$$
and 
$$\beta=\Upsilon_1 +\sum_{A=n_1+1}^{n} (h_{AA}^{n+1})^2 +\sum_{\substack{i,j=1\\i\neq j}}^n(h_{ij}^{n+1})^2+\sum_{r=n+2}^{m}~\sum_{i,j=1}^n(h_{ij}^r)^2,$$
then we derive 
\begin{equation}\label{1stinequ}
h_{11}^{n+1} h_{22}^{n+1}\ge \frac{1}{2}\biggl (\Upsilon_1 +\sum_{A=n_1+1}^{n} (h_{AA}^{n+1})^2+\sum_{\substack{i,j=1\\i\neq j}}^n(h_{ij}^{n+1})^2+\sum_{r=n+2}^{m}~\sum_{i,j=1}^n(h_{ij}^r)^2\biggr).
\end{equation}

From $(\ref{F1})$ and $(\ref{12})$ we also have
$$K(\pi_1)=c+\sum_{r=n+1}^m \biggl(h_{11}^rh_{22}^r-(h_{12}^r)^2\biggr).$$

Therefore, by combining the above two relations together, we get
$$K(\pi_1)\ge c +\frac{1}{2}\Upsilon_1+\frac{1}{2}\sum_{A=n_1+1}^{n} (h_{AA}^{n+1})^2 + \sum_{r=n+2}^m h_{11}^rh_{22}^r-\sum_{r=n+1}^m(h_{12}^r)^2 $$$$+\frac{1}{2}\sum_{\substack{i,j=1\\i\neq j}}^n(h_{ij}^{n+1})^2+\frac{1}{2}\sum_{r=n+2}^{m}~\sum_{i,j=1}^n(h_{ij}^r)^2.$$

From Lemma \ref{djei93}, it is obvious that the above inequality is identical to 
$$K(\pi_1)\ge c +\frac{1}{2}\Upsilon_1  +\frac{1}{2}\sum_{\substack {i,j=3\\i\neq j}}^n(h_{ij}^{n+1})^2+\frac{1}{2}\sum_{r=n+2}^{m}~\sum_{i,j=3}^n(h_{ij}^r)^2$$$$+\frac{1}{2}\sum_{r=n+2}^m(h_{11}^r+h_{22}^r)^2+\sum_{r=n+1}^m~\sum_{j=3}^n\biggl((h_{1j}^r)^2+ (h_{2j}^r)^2\biggr)+\frac{1}{2}\sum_{A=n_1+1}^{n} (h_{AA}^{n+1})^2.$$

Hence, the above inequality yields to
\begin{equation}\label{2ndinequ}
K(\pi_1)\ge c +\frac{1}{2}\Upsilon_1+\frac{1}{2}\sum_{\substack {i,j=3\\i\neq j}}^n(h_{ij}^{n+1})^2+\frac{1}{2}\sum_{r=n+2}^{m}~\sum_{i,j=3}^n(h_{ij}^r)^2+\frac{1}{2}\sum_{A=n_1+1}^{n} (h_{AA}^{n+1})^2.
\end{equation}

By $(\ref{dme})$ and the above equation, we obtain
$$K(\pi_1)\ge c+\tau (T_xM^n)+\frac{1}{2(n_1-1)}{\biggl(\sum_{a=1}^{n_1} h_{aa}^{n+1}\biggr)^2}-\frac{n^2}{2}||\vec H||^2-\frac{1}{2}n(n-1)c$$$$+\frac{1}{2}\sum_{\substack {i,j=3\\i\neq j}}^n(h_{ij}^{n+1})^2+\frac{1}{2}\sum_{r=n+2}^{m}~\sum_{i,j=3}^n(h_{ij}^r)^2+\frac{1}{2}\sum_{A=n_1+1}^{n} (h_{AA}^{n+1})^2.$$

The above equation can be written as
\begin{equation}\label{4thinequ}
\tau_1 (T_xN_1)-K(\pi_1)\le \frac{n^2}{2} ||\vec H||^2-\frac{n_2 \Delta f}{f}+(\frac{n^2}{2}-\frac{n}{2}-1)c-\tau_2(T_xN_2)$$$$-\frac{1}{2}\sum_{\substack {i,j=3\\i\neq j}}^n(h_{ij}^{n+1})^2-\frac{1}{2}\sum_{r=n+2}^{m}~\sum_{i,j=3}^n(h_{ij}^r)^2-\frac{1}{2}\sum_{A=n_1+1}^{n} (h_{AA}^{n+1})^2.
\end{equation}

Applying the Gauss equation on $\tau_2(T_xN_2)$, gives
\begin{equation}\label{5thinequ}
-\tau_2(T_xN_2)=-\tilde \tau_2(T_xN_2)+\frac{1}{2}\sum_{r=n+1}^m\sum_{A,B=n_1+1}^n(h_{AB}^r)^2-\frac{1}{2}\sum_{r=n+1}^m (h_{n_1+1n_1+1}^r+\cdots +h_{nn}^r)^2.
\end{equation}

In view of the above two relations, we can write
\begin{equation}\label{vbg}
\tau_1 (T_xN_1)-K(\pi_1)\le \frac{n^2}{2} ||\vec H||^2-\frac{n_2 \Delta f}{f}+(\frac{n^2}{2}-\frac{n}{2}-1)c-\tilde\tau_2(T_xN_2)$$$$-\frac{1}{2}\sum_{\substack {i,j=3\\i\neq j}}^n(h_{ij}^{n+1})^2-\frac{1}{2}\sum_{r=n+2}^{m}~\sum_{i,j=3}^n(h_{ij}^r)^2-\frac{1}{2}\sum_{A=n_1+1}^{n} (h_{AA}^{n+1})^2+\frac{1}{2}\sum_{r=n+1}^m\sum_{A,B=n_1+1}^n(h_{AB}^r)^2.
\end{equation}

Now, it clear that $(\ref{vbg})$ is equivalent to the following 
\begin{equation}\label{3rdinequ} 
\tau_1 (T_xN_1)-K(\pi_1)\le \frac{n^2}{2} ||\vec H||^2-\frac{n_2 \Delta f}{f}+(\frac{n^2}{2}-\frac{n}{2}-1)c-\tilde\tau_2(T_xN_2)$$$$-\frac{1}{2}\sum_{r=n+2}^m\sum_{a,b=3}^{n_1}(h_{ab}^{r})^2-\sum_{r=n+2}^m\sum_{a=3}^{n_1}\sum_{A=n_1+1}^{n}(h_{aA}^{r})^2$$$$-\frac{1}{2}\sum_{\substack {a,b=3\\a\neq b}}^{n_1}(h_{ab}^{n+1})^2-\sum_{a=3}^{n_1}\sum_{A=n_1+1}^{n}(h_{aA}^{n+1})^2.
\end{equation}

Hence, the inequality in $(i)$ follows directly from the above one.

If $\pi_2 \subset T_xN_2$, then put $\pi_2=Span\{e_{n_1+1}, e_{n_1+2}\}$. Now, following similar analogy like the first case, we can write
$$\biggl(\sum_{A=n_1+1}^{n}h_{AA}^{n+1}\biggr)^2=2\tau (T_xM^n)+||h||^2-n(n-1)c-\biggl(\sum_{a=1}^{n_1}h_{aa}^{n+1}\biggr)^2-2\sum_{a=1}^{n_1}~\sum_{A=n_1+1}^nh_{aa}^{n+1}h_{AA}^{n+1}.$$
Putting 
\begin{equation}\label{}
\Upsilon_2 = 2\tau (T_xM^n)- \frac{n_2-2}{n_2-1}  \biggl(\sum_{A=n_1+1}^{n}h_{AA}^{n+1}\biggr)^2
~~~~~~~~~~~~~~~~~~~~~~~~~~~~~~~~
$$$$-\biggl(\sum_{a=1}^{n_1}h_{aa}^{n+1}\biggr)^2
-2\sum_{a=1}^{n_1}\sum_{A=n_1+1}^nh_{aa}^{n+1}h_{AA}^{n+1}-n(n-1)c.\label{}
\end{equation}

Thus, from the above two equations we may write
\numberwithin{equation}{section}
\begin{equation}
\biggl(\sum_{A=n_1+1}^{n}h_{AA}^{n+1}\biggr)^2= (n_2-1) \biggl(\Upsilon_2 + ||h||^2\biggr),\label{}
\end{equation}
i.e.,
\numberwithin{equation}{section}
\begin{equation}
\biggl(\sum_{A=n_1+1}^{n}h_{AA}^{n+1}\biggr)^2= (n_2-1)\biggl (\Upsilon_2 + \sum_{a=1}^{n_1} (h_{aa}^{n+1})^2+\sum_{A=n_1+1}^{n} (h_{AA}^{n+1})^2$$$$~~~~~~~~~~+\sum_{\substack{i,j=1\\i\neq j}}^n(h_{ij}^{n+1})^2+\sum_{r=n+2}^{m}~\sum_{i,j=1}^n(h_{ij}^r)^2\biggr).\label{}
\end{equation}

Applying Lemma \ref{48} on the above equation for
$$\alpha_a=h_{AA}^{n+1},~~~ \forall~ a\in \{n_1+1, \cdots, n\}$$
and 
$$\beta=\Upsilon_2 +\sum_{a=1}^{n_1} (h_{aa}^{n+1})^2 +\sum_{\substack{i,j=1\\i\neq j}}^n(h_{ij}^{n+1})^2+\sum_{r=n+2}^{m}~\sum_{i,j=1}^n(h_{ij}^r)^2,$$
then we derive 
\begin{equation}\label{}
h_{n_1+1n_1+1}^{n+1} h_{n_1+2n_1+2}^{n+1}\ge \frac{1}{2}\biggl (\Upsilon_2 +\sum_{a=1}^{n_1} (h_{aa}^{n+1})^2+\sum_{\substack{i,j=1\\i\neq j}}^n(h_{ij}^{n+1})^2+\sum_{r=n+2}^{m}~\sum_{i,j=1}^n(h_{ij}^r)^2\biggr).
\end{equation}

From $(\ref{F1})$ and $(\ref{12})$ we also have
$$K(\pi_2)=c+\sum_{r=n+1}^m \biggl(h_{n_1+1n_1+1}^rh_{n_1+2n_1+2}^r-(h_{n_1+1n_1+2}^r)^2\biggr).$$

Therefore, by combining the above two relations we reach
$$K(\pi_2)\ge c + \sum_{r=n+2}^m h_{n_1+1n_1+1}^rh_{n_1+2n_1+2}^r-\sum_{r=n+1}^m(h_{n_1+1n_1+2}^r)^2
+\frac{1}{2}\Upsilon_2 $$$$+\frac{1}{2}\sum_{a=1}^{n_1} (h_{aa}^{n+1})^2 +\frac{1}{2}\sum_{\substack{i,j=1\\i\neq j}}^n(h_{ij}^{n+1})^2+\frac{1}{2}\sum_{r=n+2}^{m}~\sum_{i,j=1}^n(h_{ij}^r)^2.$$

Now, following similar procedure as the first case, the inequality of statement $(ii)$ follows immediately.

For the equality case, we also distinguish two cases based on whether the $2$-plane $\pi_i$ is tangent to the first factor or to the second. In statement $(\grave{i})$, we consider $\pi_1\subset T_xN_1$, then the equality holds if and only if all equalities of 
$(\ref{1stinequ})$, $(\ref{2ndinequ})$, $(\ref{4thinequ})$, $(\ref{5thinequ})$ and $(\ref{3rdinequ})$ hold. One can see that these equalities hold if and only if the following conditions are satisfied, respectively. 
\begin{itemize}
\item[(i)]$h_{11}^{n+1}+h_{22}^{n+1}=h_{33}^{n+1}=\cdots =h_{n_1n_1}^{n+1},$

\item[(ii)]$\displaystyle\sum_{r=n+2}^m(h_{11}^r+h_{22}^r)^2+\sum_{r=n+1}^m~\sum_{j=3}^n\biggl((h_{1j}^r)^2+ (h_{2j}^r)^2\biggr)=0,$

\item[(iii)]$\displaystyle\biggl(\sum_{a=1}^{n_1} h_{aa}^{n+1}\biggr)^2=\sum_{r=n+1}^m (h_{n_1+1n_1+1}^r+\cdots +h_{nn}^r)^2=0,$

\item[(iv)]$\displaystyle\sum_{r=n+2}^m\sum_{a,b=3}^{n_1}(h_{ab}^{r})^2+\sum_{r=n+2}^m\sum_{a=3}^{n_1}\sum_{A=n_1+1}^{n}(h_{aA}^{r})^2\\~~~~~~~~~~~~~~~~~~~~~~~~~~~~+\sum_{\substack {a,b=3\\a\neq b}}^{n_1}(h_{ab}^{n+1})^2+\sum_{a=3}^{n_1}\sum_{A=n_1+1}^{n}(h_{aA}^{n+1})^2=0.$
\end{itemize}

From condition $(iii)$, it is clear that $N_1\times_fN_2$ is both $\mathfrak{D}_1$-minimal and $\mathfrak{D}_2$-minimal warped product submanifold in $\tilde M^m(c)$. This implies that $N_1\times_fN_2$ is minimal in $\tilde M^m(c)$. 

Now, we are going to classify the other conditions in two categories, according to the normal vector field $r$. Firstly, if $r=n+1$, then we have
$$h_{11}^{n+1}+h_{22}^{n+1}=h_{33}^{n+1}=\cdots =h_{n_1n_1}^{n+1},$$
and
$$\sum_{j=3}^n h_{1j}^{n+1}=\sum_{j=3}^n h_{2j}^{n+1}=
\sum_{\substack {a,b=3\\a\neq b}}^{n_1}h_{ab}^{n+1}=\sum_{a=3}^{n_1}\sum_{A=n_1+1}^{n}h_{aA}^{n+1}=0.$$
Equivalently,
$$A_{e_{n+1}} =
 \begin{pmatrix}
  \mu_1 & h_{12}^{n+1} & 0 & \cdots & 0_{1n_1} & \vline & 0_{1n_1+1} & \cdots & 0_{1n} \\
  h_{21}^{n+1} & \mu_2 & 0 & \cdots & \vdots & \vline & \vdots & \cdots & \vdots  \\
  0 & 0 & \mu & \cdots & \vdots & \vline& \vdots &  \cdots& \vdots \\
  \vdots & \vdots & \vdots & \ddots & \vdots & \vline & \vdots & \cdots & \vdots \\
  0_{n_11} & 0 & 0 & \cdots & \mu & \vline & 0_{n_1n_1+1} & \cdots & 0_{n_1n} \\
\hline
  0_{n_1+11} & \cdots & \cdots & \cdots & 0_{n_1+1n_1} & \vline & h^{n+1}_{n_1+1n_1+1} & \cdots & h^{n+1}_{n_1+1n} \\
  \vdots & \ddots & \ddots & \ddots & \vdots & \vline & \vdots & \ddots & \vdots \\
  0_{n1} & \cdots & \cdots & \cdots & 0_{nn_1} & \vline & h^{n+1}_{nn_1+1} & \cdots & h^{n+1}_{nn} 
 \end{pmatrix}
,$$
$\mu=\mu_1+\mu_2.$

Secondly, if $r\in \{n+2, \cdots, m\}$, then the conditions above imply
$$h_{11}^r+h_{22}^r=\sum_{j=3}^n h_{1j}^r=\sum_{j=3}^nh_{2j}^r=\sum_{a,b=3}^{n_1}h_{ab}^{r}
=\sum_{a=3}^{n_1}\sum_{A=n_1+1}^{n}h_{aA}^{r}=0.$$

Equivalently,
$$A_{e_r} =
 \begin{pmatrix}
  h_{11}^{r} & h_{12}^{r} & 0 & \cdots & 0_{1n_1} & \vline & 0_{1n_1+1} & \cdots & 0_{1n} \\
  h_{21}^{r} & - h_{11}^{r} & 0 & \cdots & \vdots & \vline & \vdots & \cdots & \vdots  \\
  0 & 0 & 0_{33} & \cdots & \vdots & \vline& \vdots & \cdots & \vdots \\
  \vdots & \vdots & \vdots & \cdots & \vdots & \vline & \vdots &\cdots  & \vdots \\
  0_{n_11} & 0 & 0 & \cdots & 0_{n_1n_1} & \vline & 0_{n_1n_1+1} & \cdots & 0_{n_1n} \\
\hline
  0_{n_1+11} & \cdots & \cdots & \cdots & 0_{n_1+1n_1} & \vline & h^r_{n_1+1n_1+1} & \cdots & h^r_{n_1+1n} \\
  \vdots & \ddots & \ddots &\ddots  & \vdots & \vline & \vdots & \ddots & \vdots \\
  0_{n1} & \cdots & \cdots & \cdots &0_{nn_1} & \vline & h^r_{nn_1+1} & \cdots & h^r_{nn} 
 \end{pmatrix}.
$$

Obviously, the above two matrices show that $N_1\times _fN_2$ is mixed totally geodesic submanifold in $\tilde M^m(c)$.
 
Analogously, the equality sign in $(\grave{ii})$ holds if and only if the following are satisfied
$$A_{e_{n+1}} = \begin{pmatrix}
 h^{n+1}_{11} & \cdots& \cdots & h^{n+1}_{1n_1} &\vline& 0_{1n_1+1} & \cdots & \cdots & \cdots &0_{1n}\\
\vdots&\ddots&&\vdots &\vline&\vdots&\ddots&\ddots&\ddots&\vdots\\
   \vdots & &\ddots & \vdots &\vline &\vdots & \ddots &\ddots  &\ddots  &\vdots\\
   h^{n+1}_{n_11} & \cdots&\cdots & h^{n+1}_{n_1n_1} &\vline& 0_{n_1n_1+1} & \cdots & \cdots & \cdots & 0_{n_1n}\\
\hline
   0_{n_1+11} & \cdots& \cdots & 0_{n_1+1n_1} &\vline &\mu_1 & h^{n+1}_{n_1+1n_1+2} & 0 & \cdots & 0_{n_1+1n}\\
    \vdots &\ddots & \ddots & \vdots  &\vline &h^{n+1}_{n_1+2n_1+1} & \mu_2 & 0 & \cdots & \vdots \\
  \vdots &\ddots  &\ddots & \vdots &\vline& 0 & 0 & \mu & \cdots & \vdots \\
  \vdots &\ddots  &\ddots & \vdots & \vline & \vdots & \vdots & 0 &\ddots &0\\
    0_{n1} &\cdots &\cdots & 0_{nn_1} &\vline &0_{nn_1+1} & 0 & \cdots&0 & \mu 
 \end{pmatrix}
,$$
where $\mu=\mu_1+\mu_2.$

Also,
$$A_{e_{r}}=
 \begin{pmatrix}
 h^r_{11} & \cdots& \cdots & h^r_{1n_1} &\vline& 0_{1n_1+1} & \cdots & \cdots & \cdots &0_{1n}\\
\vdots&\ddots&&\vdots &\vline&\vdots&\ddots&\ddots&\ddots&\vdots\\
   \vdots & &\ddots & \vdots &\vline &\vdots & \ddots &\ddots  &\ddots  &\vdots\\
   h^r_{n_11} & \cdots&\cdots & h^r_{n_1n_1} &\vline& 0_{n_1n_1+1} & \cdots & \cdots & \cdots & 0_{n_1n}\\
\hline
   0_{n_1+11} & \cdots& \cdots & 0_{n_1+1n_1} &\vline &h_{n_1+1n_1+1}^{r} &  h_{n_1+1n_1+2}^{r}& 0 & \cdots & 0_{n_1+1n}\\
    \vdots &\ddots & \ddots & \vdots  &\vline &h_{n_1+2n_1+1}^{r} & -h_{n_1+1n_1+1}^{r} & 0 & \cdots & \vdots \\
  \vdots &\ddots  &\ddots & \vdots &\vline& 0 & 0 & 0 & \cdots & \vdots \\
  \vdots &\ddots  &\ddots & \vdots & \vline & \vdots & \vdots & 0 &\ddots &0\\
    0_{n1} &\cdots &\cdots & 0_{nn_1} &\vline &0_{nn_1+1} & 0 & \cdots&0 & 0 
 \end{pmatrix}
.$$

Clearly, $N_1\times_fN_2$ is mixed totally geodesic in $\tilde M^m(c)$. Also, it is not difficult to show that $N_1\times_fN_2$ is both $\mathfrak{D}_1$-minimal and $\mathfrak{D}_2$-minimal, which implies the minimality of $N_1\times_fN_2$ in $\tilde M^m(c)$.
\end{proof}

\section{ Answer to Chern's problem: Finding the necessary condition for warped products to be Minimal}

As answers to Problem \ref{prob6}, we therefore apply the above result (i.e., $(\grave{iii})$ from \textit{Theorem 4.1}), which give a necessary condition for a warped product submanifold to be minimal in a Riemannian space form. 
\\
So the first answer is:

\begin{corollary}\label{con11}
Let $\varphi :M^n=N_1\times _fN_2 \longrightarrow \tilde M^m(c)$ be an isometric immersion of a warped product submanifold $M^n$ into a Riemannian space form $\tilde M^m(c)$. Then, for each point $x\in M^n$ and $\pi_1 \subset T_xN_1$, we have:\\ 
\begin{equation}\label{10073}
\delta_{N_1^{n_1}}(x)+\frac{n_2 \Delta f}{f}\le \frac{1}{2}n_1(n_1+2n_2-1)c-c,
\end{equation}
\\
and if the equality holds, then $\varphi$ is minimal.

\end{corollary}

\begin{remark}

 Consider any Riemannian manifold, a corresponding answer for Chern's problem can be obtained.

\end{remark}

The second answer for Chern's problem is:

\begin{corollary}\label{con22}
Let $\varphi :M^n=N_1\times _fN_2 \longrightarrow \tilde M^m(c)$ be an isometric immersion of a warped product submanifold $M^n$ into a Riemannian space form $\tilde M^m(c)$. Then, for each point $x\in M^n$ and $\pi_2 \subset T_xN_2$, we have: \\
 
\begin{equation}\label{}
\delta_{N_2^{n_2}}(x)+  \frac{n_2 \Delta f}{f}\le \frac{1}{2}n_2(n_2+2n_1-1)c-c,
\end{equation}
\\
and if the equality holds, then $\varphi$ is minimal.

\end{corollary}

\section{Research problems based on First Chen inequality}

Due to the results of this paper, we hypothesize the following open problems
\begin{problem}\label{ama1}
Prove the first Chen inequality for warped product submanifolds in complex space forms, Sasakian space forms and Kenmotsu space forms for examples.
\end{problem}

Secondly, we ask:
\begin{problem}\label{pqm2}
Give answers to Cheren's problem for ambient spaces in the previous remark.
\end{problem}

\vskip.15in
\begin{acknowledgements}
The author would like to thank the Palestine Technical University Kadoori, PTUK, for its supports to accomplish this work.
\end{acknowledgements}


\begin{thebibliography}{110}
\bibitem{AKU17} F. R. Al-Solamy,  V.A. Khan, S. Uddin, {\textit{Geometry of warped product semi-slant submanifolds of nearly Kaehler manifolds}} Results. Math. {\bf{71}} (2017), no. 3-4, 783--799.
\bibitem{ffkk0} A. Bejancu, {\textit{CR submanifolds of a Kaehler manifold I}}, Proc. Amer. Math. Soc. {\bf{69}} (1978), 135-142.
\bibitem{pom} A. Bejancu,  {\textit{Geometry of CR-submanifolds}}, D. Reidel Publishing Company, 1986 .
\bibitem{ddyy7} R. L. Bishop, B.  O'Neill, {\textit{Manifolds of negative curvature}}, Transactions of the American Mathematical Society, {\bf{145}} (1969), 1-49.
\bibitem{66449d} B.-Y. Chen, {\textit{Some pinching and classification theorems for minimal submanifolds}}, Archiv der Math. {\bf{60}} (1993), 568-578.
\bibitem{aallr4}  B.-Y. Chen, {\textit{Relations between Ricci curvature and shape operator for submanifolds with arbitrary codimensions}},  Glasgow Math. J. {\bf{41}} (1999), 33-41.
\bibitem{2233ee} B.-Y. Chen,{\textit{Geometry of warped products as Riemannian submanifolds and related problems}}, Soochow J. Math. {\bf{28}} (2002), 125-156.
\bibitem{6677bb} B.-Y. Chen, {\textit{On isometric minimal immersions from warped products into real space forms}}, Proc. Edinburgh Math. Soc. {\bf{45}} (2002), 579-587.
\bibitem{2211gg} B.-Y. Chen, {\textit{Another general inequality for CR-warped products in complex space forms}},  Hokkaido Math. J. {\bf{32}} (2003), no. 2, 415-444.
\bibitem{yyhh88}  B.-Y. Chen, {\textit{On warped product immersions}}, Journal of Geometry, {\bf{82}} (2005), no. 1-2, 36-49.
\bibitem{55kk99} B.-Y. Chen, {\textit{$\delta$-invariants, inequalities of submanifolds and their applications: in Topics in differential Geometry}}, Editura Academiei Rom$\hat a$ne, Bucharest (2008), 29-155.
\bibitem{ssee44}   B.-Y. Chen, {\textit{A survey on geometry of warped product submanifolds}}, J. Adv. Math. Stud. {\bf{6}} (2013), no. 2, 1-43.
\bibitem{CU} B.-Y. Chen, S. Uddin, {\it Warped product pointwise bi-slant submanifolds of Kaehler manifolds}, Publ. Math. Debrecen \textbf{92} (2018), no. 1-2, 183--199.
\bibitem{8822cc}  S. S. Chern, {\textit{Minimal submanifolds in a Riemannian manifold}},  Lawrence, Kansas, 1968.  
\bibitem{aassll}  A. Mustafa, A. De and S. Uddin, {\textit{Characterization of warped product submanifolds in Kenmotsu manifolds}}, Balkan J. Geom. Appl. {\bf{20}}  (2015), no. 1, 86-97.
\bibitem{GenIneq}  A. Mustafa, C. \"Ozel, P. Linker, M. Sati, A. Pigazzini, {\textit{A general inequality for warped product CR-submanifolds of K\"ahler manifolds}}, Hacet. J. Math. Stat., 2022, https://doi.org/10.15672/hujms.1018497.
\bibitem{abd}  A. Mustafa, S. Uddin and F. R. Al-Solamy,  {\textit{Chen-Ricci inequality for warped products in Kenmotsu space forms and its applications}}, Rev. R. Acad. Cienc. Exactas Fis. Nat. Ser. A Mat. {\bf{113}} (2019), no. 4, 3585-3602. 



\bibitem{22vvuu}  J. F. Nash, {\textit{$C^1$-isometric imbeddings }}, Annals  Math. {\bf{60}} (1954), no. 3, 383-396.
\bibitem{88jj99} J. F. Nash, {\textit{The imbedding problem for Riemannian manifolds}}, Annals Math. {\bf{63}} (1956), no. 1, 20-63.
\bibitem{iijj77} B. O'Neill, {\textit{Semi-Riemannian geometry with applictions to relativity}} Academic Press, New York, 1983.
\bibitem{99mmnn} R. Osserman, {\textit{Curvature in the eighties}}, Amer. Math. Monthly {\bf{97}} (1990),  731-756.
\bibitem{SC} S. Uddin, B.-Y. Chen, F. R. Al-Solamy, {\it Warped product bi-slant immersion in Kaehler manifolds}, {\textbf{14}} (2017), no. 2, Art. 95, 11 pp.
\end{thebibliography}
\end{document}